\documentclass[12pt]{article}
\usepackage[a4paper,margin=1.3in]{geometry}
\usepackage{amssymb}
\usepackage{amsmath}
\usepackage{amsfonts}
\usepackage{amsthm}
\usepackage{mathtools}
\usepackage[all]{xy}
\usepackage{tikz}

\newcommand{\C}{\mathbb{C}}
\newcommand{\Z}{\mathbb{Z}}

\newcommand{\G}{\mathcal{G}}
\newcommand{\E}{\mathcal{E}}
\newcommand{\Odd}{\mathcal{O}}
\newcommand{\X}{\mathcal{X}}

\newtheorem{theorem}{Theorem}

\newtheorem{corollary}{Corollary}

\newtheorem{definition}{Definition}
\newtheorem{example}{Example}

\newtheorem{lemma}{Lemma}

\newtheorem{proposition}{Proposition}

\numberwithin{equation}{section}

\begin{document}

\title{\textbf{Schur ring and Codes for $S$-subgroups over $\Z_{2}^{n}$}}
\author{Ronald Orozco L\'opez}

\newcommand{\Addresses}{{
  \bigskip
  \footnotesize

  \textit{E-mail address}, \texttt{rjol1805@hotmail.com}

}}

\maketitle

\begin{abstract}
In this paper the relationship between $S$-subgroups in $\Z_{2}^{n}$ and binary codes is shown. 
If the codes used are both $P(T)$-codes and $G$-codes, then the $S$-subgroup is free. The codes
constructed are cyclic, decimated or symmetric and the $S$-subgroups obtained are free under the
action the cyclic permutation subgroup, invariants under the action the decimated permutation 
subgroup and symmetric under the action of symmetric permutation subgroup, respectively. Also
it is shows that there is no codes generating whole $\Z_{2}^{n}$ in any $\G_{n}(a)$-complete
$S$-set of the $S$-ring $\mathfrak{S}(\Z_{2}^{n},S_{n})$.
\end{abstract}
{\bf Keywords:} Schur ring, codes, binary sequences, decimation, autocorrelation\\
{\bf Mathematics Subject Classification:} 05E15,94B60,05B20

\section{Introduction}

Let $G$ be a finite group with identity element $e$ and $\C[G]$ the group algebra of all formal sums 
$\sum_{g\in G}a_{g}g$, $a_{g}\in \C$, $g\in G$. For $T\subset G$, the element $\sum_{g\in T}g$ will 
be denoted by $\overline{T}$. Such an element is also called a $\textit{simple quantity}$. 
The transpose of $\overline{T} = \sum_{g\in G}a_{g}g$ is defined as $\overline{T}^{\top} = \sum_{g\in G}a_{g}(g^{-1})$. Let $\{T_{0},T_{1},...,T_{r}\}$ be a partition of $G$ and let $S$ be the subspace
of $\C[G]$ spanned by $\overline{T_{1}},\overline{T_{2}},...,\overline{T_{r}}$.  We say that $S$ is 
a $\textit{Schur ring}$ ($S$-ring, for short) over $G$ if: 

\begin{enumerate}
\item $T_{0} = \lbrace e\rbrace$, 
\item for each $i$, there is a $j$ such that $\overline{T_{i}}^{\top} = \overline{T_{j}}$,
\item for each $i$ and $j$, we have $\overline{T_{i}}\cdot\overline{T_{j}} = \sum_{k=1}^{r}\lambda_{i,j,k}\overline{T_{k}}$, for constants $\lambda_{i,j,k}\in\C$.
\end{enumerate}

The numbers $\lambda_{i,j,k}$ are the structure constants of $S$ with respect to the linear base 
$\{\overline{T_{0}},\overline{T_{1}},...,\overline{T_{r}}\}$. The sets $T_{i}$ are called the
\textit{basic sets} of the $S$-ring $S$. Any union of them is called an $S$-sets. Thus, 
$X\subseteq G$ is an $S$-set if and only if $\overline{X}\in S$. The set of all $S$-set is closed
with respect to taking inverse and product. Any subgroup of $G$ that is an $S$-set, is called an 
$S$-\textit{subgroup} of $G$ or $S$-\textit{group}. A partition
$\{T_{0},...,T_{r}\}$ of $G$ is called \textit{Schur partition} or $S$-\textit{partition} 
if $T_{0}=\{e\}$ and if for each $i$ there is some $j$ such that 
$T_{i}^{-1}=\{g^{-1}:g\in T_{i}\}=T_{j}$. It is known that there is a 1-1 correspondence between 
$S$-rings over $G$ and $S$-partitions of $G$. By using this correspondence, in this paper we will 
refer to an $S$-ring by mean of its $S$-partition.

The concept of $S$-ring was iniciated by I. Schur in their classical paper [1] which 
was published in 1933. Later, the theory of $S$-ring was developed for Wielandt [2].
But the main objective of theory was purely group theoretical concept, especially in
problem concerning the permutations groups. In the 80s and 90s, the theory received a notable 
impulse by the study of $S$-ring over cyclic groups and their applications to the graph theory 
[3],[4],[5],[6].

With the papers [8],[9] and [11] was initiated the research of $S$-ring over the group $\Z_{2}^{n}$
and was shown the relationship between this and Hadamard matrices, perfect binary sequences and
periodic compatible binary sequences. In this paper we will show the relationship between
$S$-ring over $\Z_{2}^{n}$ and binary codes. In particular we will use codes for to construct
$S$-subgroups over $\Z_{2}^{n}$. This point of view will be shown as an alternative to the cocyclic
matrices and to the difference sets used for to research hadamard matrices, since with the 
schur rings and its generator codes will be possible to understand the structure of the special
binary sequences above.

In this paper a code $\X_{n}$ generating whole $\Z_{2}^{n}$ is found. Then the others codes for 
$S$-subgroups are constructed by using $\X_{n}$ as a base. A code $\X_{n}^{\prime}$ will be called 
$G$-codes if there exists a permutation subgroup $G$ in $Aut(\Z_{2}^{n})$ such that 
$G\X_{n}^{\prime}=\X_{n}^{\prime}$. Other types of codes studied are the $P(T)$-codes, closely
related to the partition $P(T)$ of the subset $T$ of $N=\{0,1,...,n-1\}$. In fact, the 
$S$-subgroups constructed are both $P(T)$-codes and $G$-codes. This codes generating free
$S$-subgroups in $\Z_{2}^{n}$.

This paper is organized as follows. In section 2 basic concepts of theory of codes are shown.
In section 3 properties of $P(T)$-codes are stablished. Also is shown that a $P(T)$-code generates
a free subgroup in $\Z_{2}^{n}$ and that a $G$-code generates an $S$-subgroup in the same group.
In section 4 is shown the connection between $G$-codes and $\G_{n}(a)$-complete $S$-set.
$S$-subgroups and its generator codes in $S$-rings induced by the permutation subgroups
$C_{n}$, $\Delta_{n}$, $H_{n}C_{n}$, $\Delta_{n}C_{n}$ and $H_{n}\Delta_{n}C_{n}$ of 
$Aut(\Z_{2}^{n})$ are studied in the sections 5 to 10.

\section{Some terminology of the Theory of Codes}

In [11] the following terminology related to the theory of codes can be found

An arbitrary set $\mathcal{A}$ will be called an \textit{alphabet} and its elements are called
\textit{letters}. A finite sequence of letters written in the form $s_{1}s_{2}\cdots s_{n}$,
$n\geq0$, with every $s_{i}$ in $\mathcal{A}$, is called a \textit{word}. Any subsequence of 
consecutive letters of a word is a \textit{subword}. When $n=0$ the word is
the empty word and denoted by $1_{\mathcal{A}}$. Given a word $w=s_{1}s_{2}\cdots s_{n}$, the 
number $n$ is called the \textit{lenght} of $w$ and is denoted $l(w)$. Then the empty word
$1_{\mathcal{A}}$ has lenght $0$, i.e., $l(1_{\mathcal{A}})=0$. 

Let $\mathcal{A}^{*}$ denote all finite words defined on $\mathcal{A}$ and let $\mathcal{A}^{+}$
denote all finite nonempty words on $\mathcal{A}$. $\mathcal{A}^{*}$ is equipped with an 
associative binary operation obtained by concatenating two sequences: 
$$s_{1}s_{2}\cdots s_{n}\cdot t_{1}t_{2}\cdots t_{m}=s_{1}s_{2}\cdots s_{n}t_{1}t_{2}\cdots t_{m}.$$
The empty word is the unit element with respect to this operation and consequently the sets 
$\mathcal{A}^{*}$ and $\mathcal{A}^{+}$ are a monoid and a semigroup, respectivelly.

A \textit{factorization} of a word $s\in\mathcal{A}^{*}$ is a sequence $\{s_{1},s_{2},...,s_{n}\}$ of
$n\geq0$ words in $\Sigma^{*}$ such that $s=s_{1}s_{2}\cdots s_{n}$. For a subset $X$ of 
$\mathcal{A}^{*}$, we denote by $X^{*}$ the submonoid generated by $X$,

$$X^{*}=\{x_{1}x_{2}\cdots x_{n}\vert\ n\geq0,\ x_{i}\in X\}.$$

Similarly, we denote by $X^{+}$ the subsemigroup generated by $X$,

$$X^{^{+}}=\{x_{1}x_{2}\cdots x_{n}\vert\ n\geq1,\ x_{i}\in X\}.$$

By definition, each word $s$ in $X^{*}$ admits a least one factorization $x_{1}x_{2}\cdots x_{n}$
with all $x_{i}$ in $X$. Such a factorization is called an $X$-factorization.

A monoid $M$ is called \textit{free} if it has a subset $B$ such that:
\begin{enumerate}
\item $M=B^{*}$, and
\item For all $n,m\geq1$ and $x_{1},\cdots x_{n}$, $y_{1},\cdots y_{m}\in B$ we have
\begin{equation}\label{cond_free}
x_{1}\cdots x_{n}=y_{1}\cdots y_{m}\ \Rightarrow\ n=m\ \textsl{and}\ x_{i}=y_{i}\ \textsl{for}\ i=1,2,...,n. 
\end{equation}
\end{enumerate}
From condition 1., $B$ is a generating set of $M$ and condition 2. say us that each element in $M$
has an \textit{unique} representation as a product of elements of $B$. The set $B$ satisfying 1. y 2. 
is called a base of $M$.

Let $M$ be a monoid and $B$ its generating set. We say that $B$ is \textit{minimal} generating
set if no proper subset of $B$ is a generating set. A element $x$ of $M$ is called 
\textit{indecomposable} or \textit{atomic} if it cannot be expressed in the form $x=yz$ with
$y,z\neq1$.

Now, by reformulating the condition (\ref{cond_free}) we obtain the following definition. A subset 
$\mathcal{X}\subseteq\mathcal{A}^{*}$ is a \textit{code} if it satisfies the following condition:
For all $n,m\geq1$ and $x_{1},...,x_{n}$, $y_{1},...,y_{m}\in\mathcal{X}$
\begin{equation}
x_{1}\cdots x_{n}=y_{1}\cdots y_{m}\ \Rightarrow\ n=m\ \textsl{and}\ x_{i}=y_{i}\ \textsl{for}\ i=1,2,...,n. 
\end{equation}
Note that the empty word $1_{\mathcal{A}}$ is never in a code. The following theorem shows the
equivalence between codes and free generating set

\begin{theorem}
Let $\mathcal{X}\subseteq\mathcal{A}^{*}$. Then the following conditions are equivalent:
\begin{enumerate}
\item $\mathcal{X}$ is a code,
\item $\mathcal{X}$ is a free generating set, or a base, of the monoid $\mathcal{X}^{*}$,
\item $\mathcal{X}^{*}$ is free and $\mathcal{X}$ is its minimal generating set.
\end{enumerate}
\end{theorem}

On the other hand, let $G$ be a group with $\mathcal{A}\subseteq G$. Elements of $\mathcal{A}^{*}$
represent elements of $G$ closed under concatenation and inversion. The empty word represents 
$1_{G}$, the unity of $G$. Then $\mathcal{A}^{*}$ is a subgroup of $G$. A word 
$w=s_{1}s_{2}\cdots s_{n}$ in $\mathcal{A}^{*}$ is called \textit{reduced} if $w$ contains 
no subwords $xx^{-1}$ or $x^{-1}x$ for $x\in\mathcal{A}$. A group $G$ is called a free group if 
there exists a generating set $X$ of $G$ such that every non-empty reduced group word in $X$ 
defines a non-trivial element of $G$. Let $G$ be a free group on $X$. Then the cardinality of $X$ 
is called the rank of $G$.

\section{Schur rings and codes for $S$-subgroups over $\Z_{2}^{n}$}

In this paper denote by $\Z_{2}$ the cyclic group of order 2 with elements $+$
and $-$(where + and $-$ mean 1 and $-1$ respectively). Let 
$\Z_{2}^{n}=\overset{n}{\overbrace{\Z_{2}\times \cdots \times \Z_{2}}}$. Then all $X\in\Z_{2}^{n}$ are sequences of $+$ and $-$ and will be called $\Z_{2}$-\textit{sequences} or binary sequences.
All binary sequence in $\Z_{2}^{n}$ is of the form $(x_{0},x_{1},...,x_{n-1})$. 
Let $\textbf{1}$ denote the sequence $(1,1,...,1)$. As $X^{2}=\textbf{1}$ for all $X$ in 
$\Z_{2}^{n}$, then all reduced word in $\Z_{2}^{n}$ contains no the subword $XX$. Now we will find
a code generating whole $\Z_{2}^{n}$.

We define the following subset of $\Z_{2}^{n}$
\begin{eqnarray}\label{base_principal}
\mathcal{X}_{n}=
\left\{
\begin{array}{c}
X_{0}=-+\cdots++\\
X_{1}=+-\cdots++\\
\vdots\\
X_{n-2}=++\cdots-+\\
X_{n-1}=++\cdots+-
\end{array}
\right\}
\end{eqnarray}
where each $-$ is in the $i$-th position. In the following theorem we shall show that 
$\mathcal{X}_{n}$ is a base for all $\Z_{2}^{n}$

\begin{theorem}\label{theo_base_principal}
$\mathcal{X}_{n}$ is a code for $\Z_{2}^{n}$.
\end{theorem}
\begin{proof}
As $\vert\mathcal{X}_{n}\vert=n$, then all word on $\mathcal{X}_{n}$ has the form 
$$w=X_{0}^{\epsilon_{0}}X_{1}^{\epsilon_{1}}\cdots X_{n-1}^{\epsilon_{n-1}},$$
with $\epsilon_{i}=0,1$. Thus, the number of words on $\mathcal{X}_{n}$ of lenght $k$ is 
$\binom{n}{k}$, with $k$ ranging in $[1,n-1]$. As the empty word corresponds to $\textbf{1}$
and as $X_{0}X_{1}\cdots X_{n-1}=-\textbf{1}$, the number total of words constructed 
with the codewords $X_{i}$ is $2^{n}$. Hence there exist a 1-1 correspondence between all words on
$\mathcal{X}_{n}$ and all binary sequences in $\Z_{2}^{n}$. Consequently
$\mathcal{X}_{n}^{*}=\Z_{2}^{n}$ as we announce.
\end{proof}

Let $Aut(\Z_{2}^{n})$ denote the automorphism group of $\Z_{2}^{n}$ and take $G$ by a subgroup
of $Aut(\Z_{2}^{n})$. We shall denote with $\mathfrak{S}(\Z_{2}^{n},G)$ an $S$-partition of
$\Z_{2}^{n}$ under the action of $G$. As $\mathcal{X}_{n}$ generates whole $\Z_{2}^{n}$, we wish
to find codes on $\mathcal{X}_{n}$, this is, with codewords factorizable on
$\mathcal{X}_{n}$, for $S$-subgroups of $\mathfrak{S}(\Z_{2}^{n},G)$. We start with
the following definition

\begin{definition}
Let $T=\{i_{1},i_{2},...,i_{r}\}$ be a subset of $N=\{0,1,...,n-1\}$ and let 
$P(T)=\{T_{1},...,T_{s}\}$ denote a partition on $T$. A code $\X_{n}^{\prime}$ on $\X_{n}$ is a 
$P(T)$-code if $\X_{n}^{\prime}=\{Y_{T_{1}},...,Y_{T_{s}}\}$, where 
$Y_{T_{j}}=X_{j_{1}}\cdots X_{j_{k}}$ and $T_{j}=\{j_{1},...,j_{k}\}$. 
\end{definition}

The map $T_{j}\mapsto Y_{T_{j}}$ establishes an 1-1 correspondence between the blocks $T_{j}$ of 
$T$ and the codewords $Y_{T_{j}}$ of $\X_{n}^{\prime}$. Then it is easily inferred that

\begin{theorem}\label{theo_free_group_code}
Let $\mathcal{X}_{n}^{\prime}$ be a $P(T)$-code. Then
\begin{equation}
\vert\mathcal{X}_{n}^{\prime*}\vert=2^{\vert\mathcal{X}_{n}^{\prime}\vert}.
\end{equation}
We will call to $\mathcal{X}_{n}^{\prime*}$ a $P(T)$-free group.
\end{theorem}
\begin{proof}
Let $P(T)=\{T_{1},...,T_{s}\}$ be a partition of some subset $T$ of $N$. Then all word on
$\mathcal{X}_{n}^{\prime}$ is irreducible. Hence the number of words on $\mathcal{X}_{n}^{\prime}$
of lenght $k$ is $\binom{\vert\mathcal{X}^{\prime}\vert}{k}$ and
\begin{equation*}
\vert\X_{n}^{\prime*}\vert=1+\sum_{k=1}^{\vert\X_{n}^{\prime}\vert}\binom{\vert\X_{n}^{\prime}\vert}{k}=2^{\vert\X_{n}^{\prime}\vert}
\end{equation*}
\end{proof}

\begin{corollary}
$\X_{n}$ in (\ref{base_principal}) is the only $P(T)$-code generating whole $\Z_{2}^{n}$.
\end{corollary}
\begin{proof}
$\X_{n}$ is a $P(T)$-code with $P(T)=\{\{0\},\{1\},...,\{n-1\}\}$. The statement is followed from
here.
\end{proof}

Now, we will find the number of $P(T)$-free subgroups in $\Z_{2}^{n}$

\begin{theorem}
The number of $P(T)$-free subgroup in $\Z_{2}^{n}$ is $B_{\vert T\vert+1}$
where $B_{\vert T\vert}$ are the Bell numbers.
\end{theorem}
\begin{proof}
By the correspondence is clear that the number of $P(T)$-free subgroups for any subset $T$ is
$B_{\vert T\vert}$. As $\binom{n}{\vert T\vert}$ indicates the number of $\vert T\vert$-element
subsets of an $n$-element set, then $\binom{n}{\vert T\vert}B_{\vert T\vert}$ indicates the number
of $P(T)$-free subgroups with fixed size. Assuming that $B_{0}$ is the number of empty words we
obtain $\sum_{\vert T\vert=0}^{n}\binom{n}{\vert T\vert}B_{\vert T\vert}=B_{\vert T\vert+1}$.
\end{proof}

For example, the following are all $P(T)$-free subgroup of $\Z_{2}^{3}$
\begin{eqnarray*}
&&\{X_{0},X_{1},X_{2}\}^{*},\{X_{0}X_{1},X_{2}\}^{*},\{X_{0}X_{2},X_{1}\}^{*},\{X_{1}X_{2},X_{0}\}^{*}\\
&&\{X_{0}X_{1}X_{2}\}^{*}\\
&& \{X_{0},X_{1}\}^{*},\{X_{0}X_{1}\}^{*},\{X_{0},X_{2}\}^{*},\{X_{0}X_{2}\}^{*},\{X_{1},X_{2}\}^{*},\{X_{1}X_{2}\}^{*}\\
&&\{X_{0}\}^{*},\{X_{1}\}^{*},\{X_{2}\}^{*},\\
&&\{1\}^{*}
\end{eqnarray*}

The reason for deal with $P(T)$-code will be showed now. Let
\begin{equation}
\mathcal{X}_{7}=\left\{
\begin{array}{c}
X_{0}=-++++++\\
X_{1}=+-+++++\\
X_{2}=++-++++\\
X_{3}=+++-+++\\
X_{4}=++++-++\\
X_{5}=+++++-+\\
X_{6}=++++++-\\
\end{array}
\right\}
\end{equation}
a code for $\Z_{2}^{7}$ and define the set
\begin{equation}\label{G_set_not_code}
\mathcal{X}_{7}^{\prime}=\{X_{3}X_{5}X_{6},X_{2}X_{4}X_{5},X_{1}X_{3}X_{4},X_{0}X_{2}X_{3},
X_{6}X_{1}X_{2},X_{5}X_{0}X_{1},X_{4}X_{6}X_{0}\}.
\end{equation}
on $\mathcal{X}_{7}$. It is easy to show that $\mathcal{X}_{7}^{\prime}$ is not a code. For example, 
$X_{0}X_{2}X_{5}X_{6}$ has at least two factorization in $\mathcal{X}^{\prime}$, namely
$X_{0}X_{2}X_{3}\cdot X_{3}X_{5}X_{6}$ and $X_{2}X_{4}X_{5}\cdot X_{4}X_{6}X_{0}$.
Also, $\vert\mathcal{X}_{7}^{\prime*}\vert=16$ and not $128$ as desirable. Hence 
$\mathcal{X}_{7}^{\prime*}$ is not free group. However, from theorem \ref{theo_free_group_code}
a $P(T)$-code always generates a free group.

The following theorem say us as to obtain new $P(T)$-free subgroup from old. 

\begin{theorem}
Let $\mathcal{X}_{i}$ be $P(T_{i})$-codes, $1\leq i\leq r$, in $\Z_{2}^{n}$ such that 
$T_{i}\cap T_{j}=\emptyset$, $i\neq j$, and $T_{i}\subset N$. Then
\begin{equation}
\left(\bigcup_{i=1}^{r}\mathcal{X}_{i}\right)^{*}=\prod_{i=1}^{r}\mathcal{X}_{i}^{*}.
\end{equation}
and
\begin{equation}
\vert\prod_{i=1}^{r}\mathcal{X}_{i}^{*}\vert=2^{\sum_{i=1}^{r}\vert\mathcal{X}_{i}\vert}
\end{equation}
\end{theorem}
\begin{proof}
Follows by induction on number of $T_{i}$-codes $\X_{i}$.
\end{proof}

Now we will obtain a necessary condition for the existence of an $S$-subgroup

\begin{theorem}\label{theo_basic_S_subgroup}
Let $G$ be a permutation automorphic subgroup of $Aut(\Z_{2}^{n})$ acting on some
set $\mathcal{X}$ in $\Z_{2}^{n}$. Then $\mathcal{X}^{*}$ is an $S$-subgroup in
$\mathfrak{S}(\Z_{2}^{n},G)$.
\end{theorem}
\begin{proof}
We take a word $Y_{i_{1}}Y_{i_{2}}\cdots Y_{i_{r}}$ in $\X^{\prime*}$ and a $g$ in $G$. Then
\begin{eqnarray*}
g(Y_{i_{1}}Y_{i_{2}}\cdots Y_{i_{r}})&=&g(Y_{i_{1}})g(Y_{i_{2}})\cdots g(Y_{i_{r}})\\
&=&Y_{j_{1}}Y_{j_{2}}\cdots Y_{j_{r}}.
\end{eqnarray*}
As $g$ is arbritary, then $g(Y_{i_{1}}Y_{i_{2}}\cdots Y_{i_{r}})$ is in $\mathcal{X}^{\prime*}$ 
for all $g$ in $G$. Hence $G$ defines a partition on $\X^{\prime*}$ and $\X^{\prime*}$ is
an $S$-subgroup of $\mathfrak{S}(\Z_{2}^{n},G)$.
\end{proof}

Not all $S$-subgroup is free. $\X_{7}^{\prime*}$ in (\ref{G_set_not_code}) is an $S$-subgroup of
$\mathfrak{S}(\Z_{2}^{7},C_{7})$, where $C_{7}=\left\langle C\right\rangle$ is the cyclic 
permutation automorphic subgroup of $Aut(\Z_{2}^{7})$ of order 7 with $C$ the cyclic permutation
acting on all component of some $Y$ in $\Z_{2}^{7}$. But $\X_{7}^{\prime*}$ is not a free subgroup of 
$\Z_{2}^{7}$. 

Again let $G$ be a permutation automorphic subgroup of $Aut(\Z_{2}^{n})$ and let $Y_{G}$ denote the
orbit of some $Y$ in $\Z_{2}^{n}$ under the action of $G$. From previous theorem $Y_{G}^{*}$ is an 
$S$-subgroup. We will called to $Y_{G}^{*}$ a \textit{basic} $S$-subgroup of 
$\mathfrak{S}(\Z_{2}^{n},G)$. If for a code $\X$ it is true that $\X=Y_{G}$ for some $Y$ in 
$\Z_{2}^{n}$, then we will say that $\X$ is a $G$-code. In the following sections we construct
$S$-subgroups by using $G$-codes.

\section{Schur ring $\mathfrak{S}(\Z_{2}^{n},S_{n})$}

Let $\omega(X)$ denote the Hamming weight of $X\in\Z_{2}^{n}$. Thus, $\omega(X)$ is the number of 
$+$ in any $\Z_{2}-$sequences $X$ of $\Z_{2}^{n}$. Now let $\G_{n}(k)$ be the subset of 
$\Z_{2}^{n}$ such that $\omega(X)=k$ for all $X\in\G_{n}(k)$, where $0\leq k\leq n$. 

We let $T_{i}=\G_{n}(n-i)$. It is straightforward to prove that the partition 
$\mathfrak{S}(\Z_{2}^{n},S_{n})=\{\G_{n}(0),...,\G_{n}(n)\}$ induces an $S$-partition over 
$\Z_{2}^{n}$, where $S_{n}\leq Aut(\Z_{2}^{n})$ is the permutation group on $n$ objects. From [7] 
it is know that the constant structure $\lambda_{i,j,k}$ is equal to
\begin{equation}\label{estruc_cons}
\lambda_{i,j,k}=
\begin{cases}
0&\mbox{if } i+j-k\ \mbox{is an odd number}\\
\binom{k}{(j-i+k)/2}\binom{n-k}{(j+i-k)/2} &\mbox{if } i+j-k\ \mbox{is an even number}
\end{cases}
\end{equation}

From (\ref{estruc_cons}) follows that
\begingroup\makeatletter\def\f@size{11}\check@mathfonts
\begin{equation}\label{producto}
\G_{n}(a)\G_{n}(b)=
\begin{cases}
\bigcup\limits_{i=0}^{a}\G_{n}(n-a-b+2i), & 0\leq a\leq \left[\dfrac{n}{2}\right], a\leq b\leq n-a,\\
\bigcup\limits_{i=0}^{n-a}\G_{n}(a+b-n+2i), & \left[\dfrac{n}{2}\right]+1\leq a\leq n, n-a\leq b\leq a.  
\end{cases}
\end{equation}
\endgroup

From (\ref{base_principal}) we know that $\X_{n}=\G_{n}(n-1)$. Then $\X_{n}$ is an $S_{n}$-code for
$\mathfrak{S}(\Z_{2}^{n},S_{n})$. In the following corollary we found another $S_{n}$-code for 
$\mathfrak{S}(\Z_{2}^{n},S_{n})$

\begin{corollary}
$\G_{n}(1)$ is an $S_{n}$-code for $\mathfrak{S}(\Z_{2}^{n},S_{n})$.
\end{corollary}
\begin{proof}
It is enough to take into account that $-\G_{n}(n-1)=\G_{n}(1)$.
\end{proof}

We prefer to use the $S_{n}$-code $\G_{n}(n-1)$ and not $\G_{n}(1)$ because in $\G_{n}(n-1)$
the positions of the negative components are easily obtained. Indeed, $+++-\cdot+-++=+-+-$ in 
$\G_{4}(3)$ but in $\G_{4}(1)$ we have $---+\cdot-+--=+-+-$.

Next we will see that a $G$-code is contained in any $S$-set of 
$\mathfrak{S}(\Z_{2}^{n},S_{n})$

\begin{proposition}
A $G$-code $\X$ is contained in $\G_{n}(a)$ for some $a$ ranging in $[1,n-1]$.
\end{proposition}
\begin{proof}
Take $X$ in $\X$. It is easy to note that $\omega(gX)=\omega(X)$ for all $g\in G$. 
Hence $\X\subseteq\G_{n}(a)$ for some $a$ in $[1,n-1]$.
\end{proof}

On the other hand, it is follows directly from (\ref{estruc_cons}) that $\lambda_{i,j,2k+1}=0$ if 
$i+j$ is even and $\lambda_{i,j,2k}=0$ if $i+j$ is odd. The union of all basic sets $\G_{n}(2a)$ 
in $\mathfrak{S}(\Z_{2}^{n},S_{n})$ will be denoted by $\E_{n}$ and the union of all basic sets 
$\G_{n}(2a+1)$ in $\mathfrak{S}(\Z_{2}^{n},S_{n})$ will be denoted $\Odd_{n}$. The sets $\E_{2n}$ 
and $\Odd_{2n+1}$ are subgroups of order $2^{2n-1}$ and $2^{2n}$, respectively. Then 

$$\mathfrak{S}(\E_{2n},S_{n})=\{\G_{2n}(0),\G_{2n}(2),...,\G_{2n}(2n)\}$$ 
and 
$$\mathfrak{S}(\Odd_{2n+1},S_{n})=\{\G_{2n+1}(1),\G_{2n+1}(3),...,\G_{2n+1}(2n+1)\}$$
are $S$-subgroups of $\mathfrak{S}(\Z_{2}^{2n},S_{2n})$ and $\mathfrak{S}(\Z_{2}^{2n+1},S_{2n+1})$,
respectively.

From (\ref{producto}), $\G_{2n}(n)^{2}=\bigcup_{i=0}^{n}\G_{2n}(2i)=\E_{2n}$ and 
$\G_{2n+1}(n)^{2}=\bigcup_{i=0}^{n}\G_{2n+1}(2i+1)=\Odd_{2n+1}$. Therefore, neither $\G_{4n}(2n)$ 
nor $\G_{4n+3}(2n+1)$ contains some code $\X$ generating whole $\Z_{2}^{4n}$ and $\Z_{2}^{2n+1}$,
respectively. This remark is generalized below.

From [7] is obtained the following definition

\begin{definition}
Take $\G_{n}(a)$ in $\mathfrak{S}(\Z_{2}^{n},S_{n})$. Let $\mathfrak{S}^{\prime}\subset\mathfrak{S}(\Z_{2}^{n},S_{n})$ be a set of basic sets. We will call $\mathfrak{S}^{\prime}$ a 
$\G_{n}(a)$-\textbf{complete} $S$-set if it holds
\begin{enumerate}
\item $\G_{n}(i)\G_{n}(j)\supset\G_{n}(a)$ for all $\G_{n}(i),\G_{n}(j)\in\mathfrak{S}^{\prime}$,
\item There is no $\G_{n}(b)\in\mathfrak{S}(\Z_{2}^{n},S_{n})$ such that 
$\G_{n}(b)^{2}\supset\G_{n}(a)$ and $\G_{n}(b)\G_{n}(k)\supset\G_{n}(a)$ for all 
$\G_{n}(k)\in\mathfrak{S}^{\prime}$.
\end{enumerate}
\end{definition}

A important result obtained is that there is no $\G_{n}(a)$-complete for all $n$ and all $a$

\begin{theorem}\label{the_non_exis_S_comp}
\begin{enumerate}
\item There is no $\G_{2n}(2a+1)$-complete $S$-sets in $\mathfrak{S}(\Z_{2}^{2n},S_{n})$.
\item There is no $\G_{2n+1}(2a)$-complete $S$-sets in $\mathfrak{S}(\Z_{2}^{2n+1},S_{n})$.
\end{enumerate}
\end{theorem}

In the following theorem is shown the relationship between codes generating whole 
$\Z_{2}^{n}$ and non $\G_{n}(a)$-complete $S$-sets in $\mathfrak{S}(\Z_{2}^{n},S_{n})$

\begin{theorem}
There is no a code $\X$ generating whole $\Z_{2}^{n}$ in a $\G_{n}(a)$-complete $S$-set.
\end{theorem}
\begin{proof}
Let $\mathfrak{S}^{\prime}$ denote a $\G_{2n}(2a)$-complete $S$-set. From (\ref{producto})
$$\G_{2n}(2b)^{2}=\bigcup_{i=0}^{2b}\G_{2n}(2n-4b+2i)$$
for all $\G_{2n}(2b)$ in $\mathfrak{S}^{\prime}$. Then all powers of $\G_{2n}(2b)$ will
contain basic sets $\G_{2n}(2k)$ only. Therefore the basic sets in a 
$\G_{2n}(2a)$-complete can generate the $S$-subgroup $\E_{2n}$ at the most. With a similar 
argument is shown for basic sets in $\G_{2n+1}(2a+1)$-complete $S$-sets.
\end{proof}

We finish this section showing some basic sets of $\mathfrak{S}(\Z_{2}^{n},S_{n})$ that can
to contain $G$-codes generating all $\Z_{2}^{n}$.

\begin{proposition}
$_{}$
\begin{enumerate}
\item $\G_{4n}(2n-1)^{3}=\Odd_{4n}$, $\G_{4n}(2n-1)^{4}=\E_{4n}$.
\item $\G_{4n+2}(2n-1)^{3}=\Odd_{4n+2}$, $\G_{4n}(2n-1)^{4}=\E_{4n+2}$.
\item $\G_{4n+1}(2n-2)^{3}=\E_{4n+1}$, $\G_{4n+1}(2n-2)^{4}=\Odd_{4n+1}$.
\item $\G_{4n+3}(2n)^{3}=\E_{4n+3}$, $\G_{4n+3}(2n)^{4}=\Odd_{4n+3}$.
\end{enumerate}
\end{proposition}
\begin{proof}
By using (\ref{producto}) we have
\begin{eqnarray}\label{eqn_cuadrado}
\G_{4n}(2n-1)^{2}&=&\G_{4n}(2)\cup\cdots\cup\G_{4n}(4n)\\
&\supset&\G_{4n}(2n)\nonumber.
\end{eqnarray}
As 
\begin{equation}
\G_{4n}(2n-1)\G_{4n}(2n)=\G_{4n}(1)\cup\cdots\cup\G_{4n}(4n-1)
\end{equation}
then $\G_{4n}(2n-1)^{3}=\Odd_{4n}$.
From $\G_{4n}(2n-1)\G_{4n}(2n+1)\supset\G_{4n}(0)$ and from (\ref{eqn_cuadrado}) is followed
that $\G_{4n}(2n-1)^{4}=\E_{4n}$.
\end{proof}


As $S_{n}$ induces a $S$-partition on $\Z_{2}^{n}$ is straightforward to prove that $G$ induces a
$S$-partition on $\Z_{2}^{n}$ for all $G\leq S_{n}\leq Aut(\Z_{2}^{n})$. In the following
sections we will construct $S$-subgroups by using $G$-codes in $S$-ring $\mathfrak{S}(\Z_{2}^{n},G)$.

\section{Schur ring $\mathfrak{S}(\Z_{2}^{n},C_{n})$}

Let $C$ denote the cyclic permutation on the components $+$ and $-$ of $X$ in 
$\Z_{2}^{n}$ such that
\begin{equation}\label{cir1}
C(X)=C\left( x_{0},x_{1},...,x_{n-2},x_{n-1}\right) =\left(x_{1},x_{2},x_{3},...,x_{0}\right),
\end{equation}
that is, $C(x_{i})=x_{(i+1) mod n}$. The permutation $C$ is a generator of cyclic group 
$C_{n}=\left\langle C\right\rangle$ of order $n$. Let 
$X_{C}=Orb_{C_{n}}X=\{C^{i}(X):C^{i}\in C_{n} \}$. Therefore, $C_{n}$ defines a partition in
equivalent class on $\Z_{2}^{n}$ which is an $S$-partition and this we shall denote by 
$\Z_{2C}^{n}=\mathfrak{S}(\Z_{2}^{n},C_{n})$. It is worth mentioning that this Schur ring 
corresponds to the orbit Schur ring induced by the cyclic permutation automorphic subgroup 
$C_{n}\le S_n\le Aut(\Z_2^n)$. 

On the other hand, let $X=\{x_{i}\}$ and $Y=\{y_{i}\}$ be two complex-valued sequences of period 
$n$. The periodic correlation of $X$ and $Y$ at shift $k$ is the product defined by:

\begin{equation}
\mathsf{P}_{X,Y}(k)=\sum\limits_{i=0}^{n-1}x_{i}\overline{y}_{i+k},\ k=0,1,...,n-1,
\end{equation}

where $\overline{a}$ denotes the complex conjugation of $a$ and $i+k$ is calculated modulo $n$.
If $Y=X$, the correlation $\mathsf{P}_{X,Y}(k)$ is denoted by $\mathsf{P}_{X}(k)$ and is the
autocorrelation of $X$. Obviously, 

\begin{eqnarray}
\mathsf{P}_{X}(k)&=&\overline{\mathsf{P}_{X}(n-k)},\label{symmetric_autoc}\\
\mathsf{P}_{-X}(k)&=&\mathsf{P}_{X}(k),\label{negative_autoc}\\
\mathsf{P}_{C^{i}X}(k)&=&\mathsf{P}_{X}(k)\label{cyclic_autoc},
\end{eqnarray}
for all $0\leq i\leq n-1$ and for all $X$ in $\Z_{2}^{n}$.\\

If $X$ is a $\Z_{2}$-sequence of length $n$, $\mathsf{P}_{X}(k)= 2\omega \left\{Y_{k}\right\}-n$,
where $ Y_{k}=XC^{k}X $. Also by (\ref{producto}), if $X\in \G_{n}(a)$, then

\begin{equation}
\mathsf{P}_{X}(k)=n-4a+4i_{k},
\end{equation} 

for some $0\leq i_{k}\leq a$ and $n-\mathsf{P}_{X}(k)$ is divisible by 4 for all $k$.

We know from theorem \ref{theo_base_principal} that $\G_{n}(n-1)$ is a code for 
$\mathfrak{S}(\Z_{2}^{n},S_{n})$ for all $n$. As $\G_{n}(n-1)=\{X,CX,C^{2}X,...,C^{n-1}\}=\X_{C}$,
then $\X_{C}$ is a $C_{n}$-code for $\mathfrak{S}(\Z_{2}^{n},C_{n})$. Then for to obtain 
information from each basic set $Y_{C}$ in $\mathfrak{S}(\Z_{2}^{n},C_{n})$ we must do it through 
of its $\X_{C}$-factorization with the $C_{n}$-code $\X_{C}$. The advantage of using this code lies 
in its simplicity, since each $C^{i}X$ has exactly a $-$ as its component and thereby it is 
possible to know exactly the Hamming weight of each word writing with this basis.

All word $Y$ in $\G_{n}(a)$ has the form $C^{i_{1}}XC^{i_{2}}X\cdots C^{i_{r}}X$ with length
$\vert Y\vert=r$ and with $a=n-r$. Then, every basic set in $\G_{nC}(a)$ has form
\begin{equation}
Y_{C}=\bigcup_{k=0}^{n-1}C^{i_{1}+k}XC^{i_{2}+k}X\cdots C^{i_{r}+k}X,
\end{equation}

and in this way if $Z=C^{j_{1}}XC^{j_{2}}X\cdots C^{j_{s}}X$, we have
\begin{eqnarray*}
Y_{C}Z_{C}&=&\bigcup_{k=0}^{n-1}(YC^{k}Z)_{C}\\
&=&\bigcup_{k=0}^{n-1}(C^{i_{1}}X\cdots C^{i_{r}}XC^{j_{1}+k}X\cdots C^{j_{s}+k}X)_{C}.
\end{eqnarray*}

Each word $YC^{k}Z=C^{i_{1}}X\cdots C^{i_{r}}XC^{j_{1}+k}X\cdots C^{j_{s}+k}X$ can be reduced if
exist two equal letters. Thereupon $YC^{k}Z$ decreases its length an even number. Therefore $YC^{k}Z$ 
belong to $\G_{n}(b)$ with $b=n-(r+s)+2w$ where $2w$ is the number of canceled letters. If both
$Y$ and $Z$ belongs to $\G_{n}(a)$, then $b=n-2r+2w$ and $\mathsf{P}_{Y,Z}(k)=n-4r+4w_{k}$. 

Next we will obtain the algebraic version of (\ref{symmetric_autoc}), (\ref{negative_autoc}) and
(\ref{cyclic_autoc})

\begin{proposition}
Let $Y$ denote the binary sequence $C^{i_{1}}XC^{i_{2}}X\cdots C^{i_{r}}X$. If $YC^{k}Y\in\G_{n}(a)$,
then
\begin{enumerate}
\item $YC^{n-k}Y$ and $(C^{j}X)C^{k}(C^{j}X)$ are in $\G_{n}(a)$ too. 
\item $(-Y)C^{k}(-Y)\in\G_{n}(a)$.
\end{enumerate}
\end{proposition}
\begin{proof}
1. It is clear that 
\begin{equation*}
YC^{k}Y=C^{i_{1}}XC^{i_{2}}X\cdots C^{i_{r}}XC^{i_{1}+k}XC^{i_{2}+k}X\cdots C^{i_{r}+k}X\in\G_{n}(a)
\end{equation*}
with $a=n-2r+2w$, where $2w$ are the number of canceled letters. We wish to show that the
cancellation numbers of $YC^{k}Y$ and $YC^{n-k}Y$ coincide. Suppose that $i_{j}=i_{1}+k$ for some
$j$ and some $k$. Then this implies that $n-k+i_{j}=i_{1}$ reduced module $n$. Therefore
\begin{equation*}
YC^{n-k}Y=C^{i_{1}}X\cdots C^{i_{r}}XC^{n-k+i_{1}}X\cdots C^{n-k+i_{r}}X
\end{equation*}
has the same number of cancellations as $YC^{k}Y$. Equally is proved for $(C^{j}X)C^{k}(C^{j}X)$.
2. As $C^{k}(-Y)=-C^{k}Y$, then $(-Y)C^{k}(-Y)=YC^{k}Y\in\G_{n}(a)$.
\end{proof}

Now we will show other advantage of to use the $C_{n}$-code $\X_{C}$

\begin{proposition}
For all $n\leq2$ we have
\begin{equation}
\overline{\G_{n}(n-2)}^{2}=n+2\overline{\G_{n}(n-2)}+(n-3)\overline{\G_{n}(n-4)}.
\end{equation}
\end{proposition}
\begin{proof}
All word $Y$ in $\G_{n}(n-2)$ has the form $C^{i}XC^{j}X$ with $i<j$. Then 
$YC^{k}Y=C^{i}XC^{j}XC^{i+k}XC^{j+k}X$ and there exist a $k$ such that either $i+k=j$ or $j+k=i$
and for all the remaining values of $k$ we have that $C^{i}XC^{j}XC^{i+k}XC^{j+k}X$ is a reduced
word. As $YC^{k}Y$ and $YC^{n-k}Y$ are in $\G_{n}(a)$ for some $a$, then $Y_{C}^{2}$ contains 
$2$ words in $\G_{n}(n-2)$, $n-3$ words in $\G_{n}(n-4)$ and the trivial word in $\G_{n}(n)$.
\end{proof}

On the other hand, let 
\begin{equation}
F_{d}(\Z_{2}^{n})=\bigcup_{\vert X\vert=d}X.
\end{equation}
Clearly $d$ divides to $n$ and the $X\in F_{d}(\Z_{2}^{n})$ have the form $X=(Y,Y,...,Y)$, with 
$Y\in\Z_{2}^{d}$. Then $F_{d}(\Z_{2C}^{n})=\bigcup_{\vert X_{C}\vert=d}X_{C}$ is an $S$-set of
$\Z_{2C}^{n}$, for each $d\vert n$. When $d=n$, we will to say that $C_{n}$ acts freely on 
$X_{C}$ and we denote $F_{n}(\Z_{2C}^{n})$ as $F(\Z_{2C}^{n})$. When $d<n$, we will to say that
$C_{n}$ don't act freely on $X_{C}$ and let $\widehat{F}(\Z_{2C}^{n})$ denote the set of the 
$X_{C}$ which are not frees under the action of $C_{n}$, namely

\begin{equation}\label{Z_not_libre}
\widehat{F}(\Z_{2C}^{n})=\bigcup_{d\mid n,d<n}F_{d}(\Z_{2C}^{n}).
\end{equation}

Therefore, 
\begin{eqnarray}\label{Z_libre}
\Z_{2C}^{n}&=&F(\Z_{2C}^{n})\cup \widehat{F}(\Z_{2C}^{n})\nonumber\\
&=&\bigcup_{d\mid n}F_{d}(\Z_{2C}^{n}). 
\end{eqnarray}

The set $F_{d}(\Z_{2}^{n})$ is constructed with codewords in
\begin{equation}
\X_{F,d}=\{A_{0,d}X,A_{1,d}X,\cdots,A_{d-1,d}X\}
\end{equation}
where $A_{i,d}X=C^{i}XC^{i+d}X\cdots C^{i+kd}X$ for $i=0,1,2,...,d-1$ with $k=\frac{n}{d}-1$, 
$X\in\G_{n}(n-1)$ and 

$$\{i,i+d,i+2d,\cdots,i+kd\}$$

is an arithmetic progression. Then $\X_{F,d}$ is a $P(T)$-code with 
$$P(T)=\{\{i,i+d,i+2d,...,i+kd\}:\ i=0,1,...,d-1\}$$
and any word in $\X_{F,d}^{*}$ has the form
\begin{equation}
Y=A_{0,d}^{\epsilon_{0}}XA_{1,d}^{\epsilon_{1}}X\cdots A_{d-1,d}^{\epsilon_{d-1}}X
\end{equation}
$\epsilon_{i}=0,1$. It is clear that $C^{i}A_{j,d}X=A_{i+j,d}X$. Hence $\X_{F,d}$ is a
$C_{n}$-code.

If $d=1$, then $k=n-1$, $i=0$ and 
\begin{equation}
\X_{F,1}=\{A_{0,1}X\}=\{XCXC^{2}X\cdots C^{n-1}X\}
\end{equation}
is a code with a codeword. As $XCXC^{2}X\cdots C^{n-1}X=-\textbf{1}$, then 
$\X_{F,1}^{*}=\{\textbf{1},-\textbf{1}\}$ for all $n$.

If $d=n$, then $k=0$, $i=0,1,\cdots,n-1$ and 
\begin{equation}
\X_{F,n}=\{A_{0,n}X,A_{1,n}X,\cdots,A_{n-1,n}X\}=\{X,CX,C^{2}X,\cdots,C^{n-1}X\}
\end{equation}
is an code with exactly $n$ codewords and $\X_{F,d}=\X_{C}$. Let $\vert\X_{F,d}\vert$ be the rank
of $\X_{F,d}$. We then note that $1<\vert\X_{F,d}\vert<n$ for $1<d<n$.

Now we will see the relationship between free subgroup $\X_{F,d}^{*}$ and the sets 
$F_{d}(\Z_{2}^{n})$

\begin{theorem}
$\X_{F,d}^{*}$ is a subgroup of $\Z_{2}^{n}$ of order $2^{d}$ with 
$\X_{F,d}^{*}=\bigcup_{r\vert d}F_{d}(\Z_{2}^{n})$. We will donote this subgroup with 
$\mathbb{G}_{d}(n)$.
\end{theorem}
\begin{proof}
As $\X_{F,d}$ is a $P(T)$-code, then  $\mathbb{G}_{d}(n)$ is a subgroup of $\Z_{2}^{n}$ of order
$2^{d}$ for all divisor $d$ of $n$. Then we only will show that $\mathbb{G}_{d}(n)$ has the 
desired structure. If $d$ is a prime divisor of $n$, then all words of $\mathbb{G}_{d}(n)$ are 
in $F_{d}(\Z_{2}^{n})$, except for $\textbf{1}$ and $-\textbf{1}$. Hence
\begin{equation}
\mathbb{G}_{d}(n)=F_{1}(\Z_{2}^{n})\cup F_{d}(\Z_{2}^{n})
\end{equation}
with
\begin{eqnarray*}
-\textbf{1}&=&\prod_{i=0}^{d-1}A_{i,d}X\\
\textbf{1}&=&(A_{i,d}X)^{2}\ for\ all\ i.
\end{eqnarray*} 
Suppose that $d$ is no prime. Then
\begin{equation}
A_{i,d}XA_{i+r,d}X\cdots A_{i+(\frac{d}{r}-1)r,d}X
\end{equation}
is contained in $F_{r}(\Z_{2}^{n})$, $r\vert d$ and $i=0,1,...,r-1$. Therefore 
$\mathbb{G}_{d}(n)=\bigcup_{r\vert d}F_{r}(\Z_{2}^{n})$.
\end{proof}

\begin{corollary}
$\mathbb{G}_{dC}(n)$ is an $S$-subgroup of $\mathfrak{S}(\Z_{2}^{n},C_{n})$.
\end{corollary}
\begin{proof}
The group $\left\langle C_{n}\right\rangle$ defines a partition on each $F_{r}(\Z_{2}^{n})$ in
$\mathbb{G}_{d}(n)$ and hence we obtain the desired statement.
\end{proof}

\begin{example}
The subgroup $\mathbb{G}_{3}(9)$ of $\Z_{2}^{9}$ is given by 
\begin{eqnarray*}
\mathbb{G}_{3}(9)&=&F_{1}(\Z_{2}^{9})\cup F_{3}(\Z_{2}^{9})\\
&=&\{-\textbf{1},\textbf{1}\}\cup\{XC^{3}XC^{6}X,\ CXC^{4}XC^{7}X,\ C^{2}XC^{5}XC^{8}X,\\
&&XCXC^{3}XC^{4}XC^{6}XC^{7}X,\ XC^{2}XC^{3}XC^{5}XC^{6}XC^{8}X,\\
&&CXC^{2}XC^{4}XC^{5}XC^{7}XC^{8}X\}\\
&=&\{-\textbf{1},\textbf{1},-++-++-++,++-++-++-,\\
&&+-++-++-+,-+--+--+-,--+--+--+,\\
&&+--+--+--\}.
\end{eqnarray*}
And the $S$-subgroup $\mathbb{G}_{3C}(9)$ of $\mathfrak{S}(\Z_{2}^{9},C_{9})$ is given by
\begin{eqnarray*}
\mathbb{G}_{3C}(9)&=&\{\textbf{1},-\textbf{1},(-++-++-++)_{C},(-+--+--+-)_{C}\}.
\end{eqnarray*}
\end{example}

\begin{theorem}
$\mathbb{G}_{d}(n)\subseteq\bigcup_{a=0}^{d}\G_{n}\left(\frac{na}{d}\right)$ with equality only for
$d=1,n$.
\end{theorem}
\begin{proof}
Follows from $F_{d}(\Z_{2}^{n})\subseteq\G_{n}\left(\frac{n}{d}\right)\cup\G_{n}\left(n-\frac{n}{d}\right)$.
\end{proof}

Now we show the lattice of $S$-subgroups $\mathbb{G}_{dC}(60)$ of 
$\mathfrak{S}(\Z_{2}^{60},C_{60})$ ordered by inclusion.

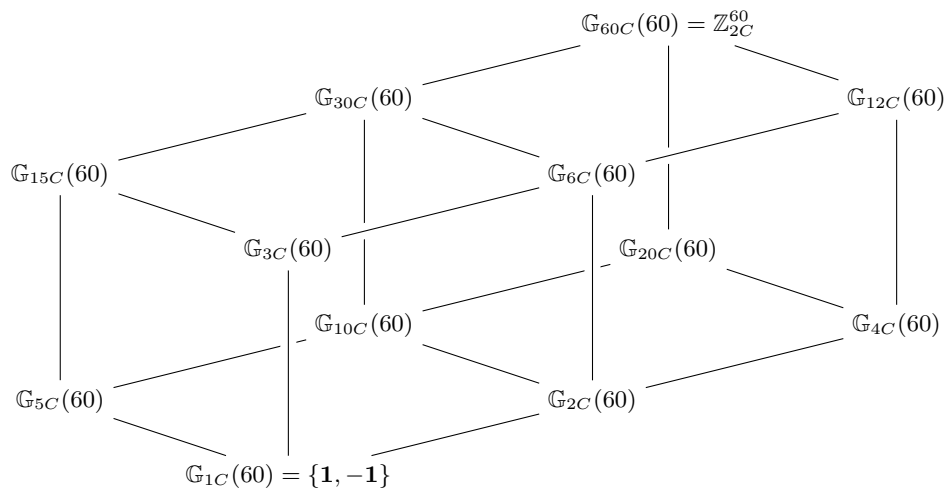
\begingroup\makeatletter\def\f@size{9}\check@mathfonts
\begin{center}
\begin{tikzpicture}
  \node (max_med) at (1,3) {$\mathbb{G}_{30C}(60)$};
  \node (a) at (-3,2) {$\mathbb{G}_{15C}(60)$};
  \node (b) at (1,0) {$\mathbb{G}_{10C}(60)$};
  \node (c) at (4,2) {$\mathbb{G}_{6C}(60)$};
  \node (d) at (-3,-1) {$\mathbb{G}_{5C}(60)$};
  \node (e) at (0,1) {$\mathbb{G}_{3C}(60)$};
  \node (f) at (4,-1) {$\mathbb{G}_{2C}(60)$};
  \node (min) at (0,-2) {$\mathbb{G}_{1C}(60)=\{\textbf{1},-\textbf{1}\}$};
  \node (max) at (5,4) {$\mathbb{G}_{60C}(60)=\Z_{2C}^{60}$};
  \node (g) at (5,1) {$\mathbb{G}_{20C}(60)$};
  \node (h) at (8,3) {$\mathbb{G}_{12C}(60)$};
  \node (i) at (8,0) {$\mathbb{G}_{4C}(60)$};
  \draw (min) -- (d) -- (a) -- (max_med) -- (b) -- (f) -- (i) -- (h) -- (max)
  (e) -- (min) -- (f) -- (c) -- (max_med) -- (max)
  (d) -- (b) -- (g) -- (max)
  (c) -- (h)
  (i) -- (g);
  \draw[preaction={draw=white, -,line width=6pt}] (a)-- (e) -- (c) -- (max_med)
  (c) -- (h)
  (f) -- (c)
  (e) -- (min);
\end{tikzpicture}
\end{center}
\endgroup

We finish this section we provide other proof to the Theorems 6 and 7 in [7]. We start with 
the lemma

\begin{lemma}
\begin{equation}
XC^{n}X\in F_{n}(\Z_{2}^{2n})
\end{equation}
for all $X$ in $\X_{C}=\G_{2n}(2n-1)$
\end{lemma}
\begin{proof}
Clearly $XC^{n}X$ is in $F_{n}(\Z_{2}^{2n})$ when $X=-+++\cdots+++$. As 
$$(C^{i}X)C^{n}(C^{i}X)=C^{i}(XC^{n}X)$$
for any other codeword $C^{i}X$ in the code $\X_{C}$, then 
$C^{i}(XC^{n}X)\in F_{n}(\Z_{2}^{2n})$ for all $1\leq i\leq 2n-1$, since 
$\left\langle C_{n}\right\rangle$ defines a partition on $F_{n}(\Z_{2}^{2n})$.
\end{proof}

\begin{theorem}
If $Y_{C}\in F(\Z_{2C}^{2n})$, then $Y_{C}^{2}\setminus\{\textbf{1}\}\not\in F(\Z_{2C}^{2n})$.
\end{theorem}
\begin{proof}
Take $Y=C^{i_{1}}X\cdots C^{i_{r}}X$ in $F(\Z_{2}^{2n})$. Then
\begin{eqnarray*}
YC^{n}Y&=&C^{i_{1}}X\cdots C^{i_{r}}XC^{i_{1}+n}X\cdots C^{i_{r}+n}X\\
&=&C^{i_{1}}(XC^{n}X)\cdots C^{i_{r}}(XC^{n}X).
\end{eqnarray*}
From the above lemma $XC^{n}X$ is in $F_{n}(\Z_{2}^{2n})$ for all $X$ in $\X_{C}$. Then
$YC^{n}Y\in F_{n}(\Z_{2}^{2n})$ for all $Y$ in $F(\Z_{2}^{n})$ and
$$Y_{C}^{2}\setminus\{\textbf{1}\}\notin F(\Z_{2C}^{2n})$$
as we promised to show.
\end{proof}

\begin{lemma}
\begin{equation}
XC^{k}X\notin F_{d}(\Z_{2}^{2n+1})
\end{equation}
for no $X$ in $\X_{C}=\G_{2n+1}(2n)$ and for no $k$ ranging in $[1,2n]$, $d<2n+1$.
\end{lemma}
\begin{proof}
Take $X=-++\cdots++$ in $\X_{C}$. Then $XC^{k}X=-\overset{2n-k}{\overbrace{+\cdots+}}-\overset{k-1}{\overbrace{+\cdots+}}$. If we want $k-1=2n-k$, then $k=\frac{2n+1}{2}$, which is 
not possible. Hence $XC^{k}X$ is no contained in $F_{d}(\Z_{2}^{2n+1})$, $d<2n+1$. As 
$C^{i}(XC^{k}X)=(C^{i}X)C^{k}(C^{i}X)$, it is followed the statement.
\end{proof}

\begin{theorem}
If $X_{C}\in F(\Z_{2C}^{2n+1})$, then $X_{C}^{2}\setminus\{\textbf{1}\}\in F(\Z_{2C}^{2n+1})$.
\end{theorem}
\begin{proof}
Let $Y=C^{i_{1}}XC^{i_{2}}X\cdots C^{i_{r}}X$ such that all the $i_{j}$ are not in arithmetic
progression. If $YC^{k}Y=C^{i_{1}}(XC^{k}X)C^{i_{2}}(XC^{k}X)\cdots C^{i_{r}}(XC^{k}X)$ is
contained in some $F_{d}(\Z_{2}^{2n+1})$, $d\vert(2n+1)$, $d<2n+1$, then $XC^{k}X$ must be 
contained in $F_{d}(\Z_{2}^{2n+1})$, but is not possible by the previous lemma.
\end{proof}

\section{Schur ring $\mathfrak{S}(\Z_{2}^{n},\Delta_{n})$}

Let $\delta_{a}\in S_{n}$ act on $X\in\Z_{2}^{n}$ by decimation, that is, 
$\delta_{a}(x_{i})=x_{ai(\mod n)}$ for all $x_{i}$ in $X$, $(a,n)=1$ and let $\Delta_{n}$ denote 
the set of this $\delta_{a}$. The set $\Delta_{n}$ is a group of order $\phi(n)$ isomorphic to
$\Z_{n}^{*}$, the group the units of $\Z_{n}$, where $\phi$ is called the Euler totient function. 
Clearly $\mathfrak{S}(\Z_{2}^{n},\Delta_{n})$ is an $S$-partition of $\Z_{2}^{n}$. In this section,
we will construct $\Delta_{n}$-codes for $S$-subgroups of $\mathfrak{S}(\Z_{2}^{n},\Delta_{n})$.

We will use the commutation relation $C^{i}\delta_{a}=\delta_{a}C^{ia}$ for to prove all of results
in this section. We begin for show that $\G_{n}(n-1)$ is partitioned in three equivalence class

\begin{proposition}
\begin{equation}
\Delta_{n}\G_{n}(n-1)=\{X\}\cup\mathcal{X}_{\Z_{n}^{*}}\cup\mathcal{X}_{\Z_{n}\setminus\Z_{n}^{*}}
\end{equation}
where $X=-++\cdots++$ and
\begin{eqnarray}
\mathcal{X}_{\Z_{n}^{*}}&=&\{C^{a_{1}}X,C^{a_{2}}X,...,C^{a_{\phi(n)}}X:\ (a_{i},n)=1\}\\
\mathcal{X}_{\Z_{n}\setminus\Z_{n}^{*}}&=&\{C^{d_{1}}X,C^{d_{2}}X,...,C^{d_{r}}X:\ (d_{i},n)\neq1\}
\end{eqnarray}
$r=n-\phi(n)-1$.
\end{proposition}
\begin{proof}
It is very easy to see that $X=-++\cdots++$ is fixed under the action of $\Delta_{n}$. Also, as 
$\delta_{a}C^{i}X=C^{a^{-1}i}\delta_{a}X=C^{a^{-1}i}X$, then $\Delta_{n}\mathcal{X}_{\Z_{n}^{*}}=\mathcal{X}_{\Z_{n}^{*}}$ and 
$\Delta_{n}\mathcal{X}_{\Z_{n}\setminus\Z_{n}^{*}}=\mathcal{X}_{\Z_{n}\setminus\Z_{n}^{*}}$.
\end{proof}

As each $C^{i}X$ in $\mathcal{X}_{\Z_{n}^{*}}$ or in $\mathcal{X}_{\Z_{n}\setminus\Z_{n}^{*}}$ 
is atomic, then $\mathcal{X}_{\Z_{n}^{*}}$ and $\mathcal{X}_{\Z_{n}\setminus\Z_{n}^{*}}$ are 
$\Delta_{n}$-code and hence $\mathcal{X}_{\Z_{n}^{*}}^{*}$ and 
$\mathcal{X}_{\Z_{n}\setminus\Z_{n}^{*}}^{*}$ are $S$-subgroups in
$\mathfrak{S}(\Z_{2}^{n},\Delta_{n})$ with $\vert\mathcal{X}_{\Z_{n}^{*}}^{*}\vert=2^{\phi(n)}$ and 
$\vert\mathcal{X}_{\Z_{n}\setminus\Z_{n}^{*}}^{*}\vert=2^{n-\phi(n)-1}$.

On the other hand, let 
$$(\mathsf{P}_{Y}(0),\mathsf{P}_{Y}(1),...,\mathsf{P}_{Y}(n-1))$$
denote the autocorrelation vector of $Y$ in $\Z_{2}^{n}$ and let $\mathfrak{A}(\Z_{2}^{n})$
denote the set of all this. Let $X_{1}+X_{2}+\cdots +X_{n}=a$ denote the plane in $\Z^{n}$ in the 
indeterminates $X_{i}$, $i=1,2,...,n$ and let 
$\theta:\Z_{2}^{n}\rightarrow\mathfrak{A}(\Z_{2}^{n})$ be the map defined by 
$\theta(Y)=(\mathsf{P}_{Y}(0),\mathsf{P}_{Y}(1),\dots ,\mathsf{P}_{Y}(n-1))$.

The decimation group $\Delta_{n}$ do not alter the set of values which $\mathsf{P}_{X}(k)$ takes
on, but merely the order in which they appear, i.e., if $Y=\delta_{a}X$ then 
$\mathsf{P}_{Y}(k)=\mathsf{P}_{X}(ka)$. Therefore, we have the commutative diagram
\begin{equation}\label{diagram}
\xymatrix{
 \Z_{2}^{n} \ar[d]^{\theta} \ar[r]^{\delta_{r}} & \Z_{2}^{n} \ar[d]^{\theta}\\
   \mathfrak{A}(\Z_{2}^{n}) \ar[r]^{\delta_{r}} & \mathfrak{A}(\Z_{2}^{n}) 
}
\end{equation}
and $\theta \circ \delta_{r} = \delta_{r}\circ \theta.$

Let $Y\in\Z_{2}^{n}$ such that $\theta(Y)=(n,d,d,...,d)$. Such a binary sequence is known as 
binary sequence with $2$-levels autocorrelation value and are important by its applications on
telecommunication. We want to construct a $\Delta_{n}$-code for some $S$-subgroup $H$ of 
$\mathfrak{S}(\Z_{2}^{n},\Delta_{n})$ containing such $Y$. From (\ref{diagram}) is followed that 
$\theta(Y)=\delta_{a}\theta(Y)=\theta(\delta_{a}Y)$, for all $\delta_{a}\in\Delta_{n}$. Hence $Y$ 
and $\delta_{a}Y$ have the same autocorrelation vector. For $Y$ fullfilling $\delta_{a}Y=Y$ for some 
$\delta_{a}$ in $\Delta_{n}$ we have the following definition

\begin{definition}
Let $a$ be a unit in $\Z_{n}^{*}$. A word $Y$ in $\Z_{2}^{n}$ is $\delta_{a}$-\textbf{invariant}
if $\delta_{a}Y=Y$. Denote by $\mathbb{I}_{n}(a)$ the set of these $Y$.
\end{definition}

If $Y$ is in $\mathbb{I}_{n}(a)$, then $\delta_{r}Y$ is in $\mathbb{I}_{n}(a)$, too. Also
$\delta_{a}(YZ)=\delta_{a}Y\delta_{a}Z=YZ$ for all $Y,Z$ in $\mathbb{I}_{n}(a)$. Then
$\mathbb{I}_{n}(a)$ is an $S$-subgroup of $\mathfrak{S}(\Z_{2}^{n},\Delta_{n})$. Now, we shall see
that all factorization of words in $\mathbb{I}_{n}(a)$ is relationated with cyclotomic coset
of $a$ module $n$. First, we have the following definition

\begin{definition}
Let $a$ relative prime to $n$. The cyclotomic coset of $a$ module $n$ is defined by
\begin{equation*}
\mathsf{C}_{s}=\{s,sa,sa^{2},\cdots,sa^{t-1}\}.
\end{equation*}
where $sa^{t}\equiv s\mod n$. A subset $\{s_{1},s_{2},\dots ,s_{r}\}$ of 
$\Z_{n}$ is called complete set of representatives of cyclotomic coset of $a$ modulo $n$ if 
$\mathsf{C}_{i_{1}}$,$\mathsf{C}_{i_{2}}$,..., $\mathsf{C}_{i_{r}}$ are distinct and are a 
partition of $\Z_{n}$. 
\end{definition}

Take $Y=C^{i_{1}}XC^{i_{2}}X\cdots C^{i_{r}}X$ in $\mathbb{I}_{n}(a)$ with $X=-++\cdots++$. 
We want $\delta_{a}Y=Y$. Then
\begin{eqnarray*}
\delta_{a}Y&=&\delta_{a}C^{i_{1}}X\delta_{a}C^{i_{2}}X\cdots\delta_{a}C^{i_{r}}X\\
&=&C^{i_{1}a^{-1}}\delta_{a}XC^{i_{2}a^{-1}}\delta_{a}X\cdots C^{i_{r}a^{-1}}\delta_{a}X\\
&=&C^{i_{1}a^{-1}}XC^{i_{2}a^{-1}}X\cdots C^{i_{r}a^{-1}}X
\end{eqnarray*}
since $\delta_{a}X=X$. As must be $\delta_{a}Y=Y$, then $i_{k}=a^{-1}i_{j}$ or $i_{j}=ai_{k}$ for
$1\leq k,j\leq r$. Let $\mathsf{C}_{s}X$ denote the word $C^{s}XC^{sa}X\cdots C^{sa^{t_{s}-1}}X$.
Then all $Y$ in $\mathbb{I}_{n}(a)$ has the form 
$Y=\mathsf{C}_{s_{1}}^{\epsilon_{1}}X\mathsf{C}_{s_{2}}^{\epsilon_{r}}X\cdots\mathsf{C}_{s_{r}}^{\epsilon}X$, with $\epsilon_{i}=0,1$. As $\mathbb{I}_{n}(a)$ is an $S$-subgroup in 
$\mathfrak{S}(\Z_{2}^{n},\Delta_{n})$, $\delta_{r}\mathsf{C}_{s_{i}}X=\mathsf{C}_{s_{j}}X$ and

\begin{equation}\label{alpha_invariante}
\mathcal{X}_{\mathbb{I}(a)}=\{X,\mathsf{C}_{s_{1}}X,\ \mathsf{C}_{s_{2}}X,...,\ \mathsf{C}_{s_{r}}X\}
\end{equation}

is a $\Delta_{n}$-code for $\mathbb{I}_{n}(a)$. Also $\X_{\mathbb{I}(a)}$ is a $P(T)$-code
with 
$$P(T)=\{\{0\},\mathsf{C}_{s_{1}},\mathsf{C}_{s_{2}},...,\mathsf{C}_{s_{r}}\}$$
and $\{s_{1},s_{2},...,s_{r}\}$ a complete set of representatives. Hence $\X_{\mathbb{I}(a)}^{*}$ 
has order $2^{r+1}$, where $r$ is the number of cyclotomic cosets of $a$ module $n$

In the table 1, binary sequences with $2$-level autocorrelation values with their
respective $\delta_{a}$-invariants $S$-subgroups are shown

\begin{table}[ht]
\caption{Binary sequences with $2$-level autocorrelation values}
\centering
\begin{tabular}{c c c}
\hline\hline
Sequences & Autocorrelation vector & $S$-subgroup\\
\hline
$\mathsf{C}_{0}X=X\in\G_{n}(n-1)$ & $(n,n-4,...,n-4)$ & $\mathbb{I}_{n}(a)$ for all $a\in\Z_{n}^{*}$\\
$\mathsf{C}_{2}X=CXC^{2}XC^{4}X$ & $(7,-1,...,-1)$ & $\mathbb{I}_{7}(3)$\\
$\mathsf{C}_{3}X=CXC^{3}XC^{4}XC^{5}XC^{9}X$ & $(11,-1,...,-1)$ & $\mathbb{I}_{11}(3)$\\
$\mathsf{C}_{0}X\mathsf{C}_{3}X=XCXC^{3}XC^{9}X$ & $(13,1,...,1)$ & $\mathbb{I}_{13}(3)$\\
$\mathsf{C}_{0}X\mathsf{C}_{5}X\mathsf{C}_{7}X\mathsf{C}_{10}X\mathsf{C}_{11}X$ & $(15,-1,...,-1)$ & $\mathbb{I}_{15}(4)$\\
\hline\hline
\end{tabular}
\end{table}

On the other hand, we have the following theorem

\begin{theorem}
If $\left\langle b\right\rangle$ is a subgroup of $\left\langle a\right\rangle$, then 
$\mathbb{I}_{n}(a)\leq\mathbb{I}_{n}(b)$.
\end{theorem}
\begin{proof}
Let $\mathsf{C}_{1}^{a}$ and $\mathsf{C}_{1}^{b}$ denote the classes $\{1,a,a^{2},...,a^{t-1}\}$
and $\{1,b,b^{2},...,b^{s-1}\}$. By hypothesis
$\left\langle b\right\rangle\leq\left\langle a\right\rangle$, then $\mathsf{C}_{1}^{b}\subseteq\mathsf{C}_{1}^{a}$. Hence there exists $y_{i}$ in $\left\langle a\right\rangle$ such that
$$\mathsf{C}_{1}^{a}=\mathsf{C}_{1}^{b}\cup y_{1}\mathsf{C}_{1}^{b}\cup\cdots\cup y_{k}\mathsf{C}_{1}^{b},$$
and $k=[\left\langle a\right\rangle:\left\langle b\right\rangle]$. Then is follows that
$$\mathsf{C}_{s}^{a}=\mathsf{C}_{s}^{b}\cup y_{1}\mathsf{C}_{s}^{b}\cup\cdots\cup y_{k}\mathsf{C}_{s}^{b}.$$
Therefore $\vert\X_{\mathbb{I}_{n}(a)}\vert\leq\vert\X_{\mathbb{I}_{n}(b)}\vert$ and
$\mathbb{I}_{n}(a)\leq\mathbb{I}_{n}(b)$.
\end{proof}

We finish this section constructing some $\delta_{a}$-invariants $S$-subgroups

\begin{proposition}\label{prop_menos_uno_inv}
$_{}$
\begin{enumerate}
\item $\mathbb{I}_{2n+1}(2n)=\{X,\ \mathsf{C}_{q}X:\ q\in\{1,2,...,n\}\}^{*}$, $\mathsf{C}_{q}=\{q,2n+1-q\}$
\item $\mathbb{I}_{2n}(2n-1)=\{X,\ C^{n}X,\ \mathsf{C}_{q}X:\ q\in\{1,2,...,n-1\}\}^{*}$, $\mathsf{C}_{q}=\{q,2n-q\}$.
\end{enumerate}
\end{proposition}
\begin{proof}
We proof $1$. The proof of $2$ it is analogous. We note that
$$(2n)^{2}=(2n+1-1)^{2}=(2n+1)^{2}-2(2n+1)+1\equiv1\mod(2n+1),$$
then $\mathsf{C}_{1}=\{1,2n\}$. As $2n+1\nmid2n-1$ and $q<2n+1$, then $2nq\not\equiv q\mod(2n+1)$
and the $\mathsf{C}_{q}=\{q,2nq\}$ are cyclotomic cosets of $2n$ module $2n+1$. Finally, it is easy 
to note that $2nq$ is congruent to $2n+1-q$ module $2n+1$, 
$$2nq-2n-1+q=(2n+1)(q-1)\equiv0\mod(2n+1).$$
\end{proof}

\begin{proposition}
Let $2p+1$ be an prime number with $p$ an odd prime number. The $S$-subgroups invariants in
$\Z_{2}^{2p+1}$ are $\mathbb{I}_{2p+1}(x)$, $\mathbb{I}_{2p+1}(y)$ and $\mathbb{I}_{2p+1}(2p)$,
where $x$ is a primitive root module $2p+1$ and $y$ is not neither primitive root module $2p+1$
nor $2p$.
\end{proposition}
\begin{proof}
Let $P=\{x_{1},x_{2},...,x_{t}\}$ denote the set of primitive roots module $2p+1$. Then
$$\mathbb{I}_{2p+1}(x_{1})=\mathbb{I}_{2p+1}(x_{2})=\cdots=\mathbb{I}_{2p+1}(x_{t}).$$
As $\vert\left\langle x_{i}\right\rangle\vert=2p$ for any $x_{i}\in P$, then 
$\left\langle x_{i}\right\rangle$ has exactly a subgroup of order $2$ and a subgroup of order $p$.
Therefore there exist $S$-subgroups invariants $\mathbb{I}_{2p+1}(y)$ and $\mathbb{I}_{2p+1}(2p)$ 
where $y\in P^{c}\setminus\{2p\}$, with $P^{c}$ the complement of $P$ in $\Z_{2p+1}^{*}$.
\end{proof}

\section{Schur ring $\mathfrak{S}(\Z_{2}^{n},H_{n})$}

We note by $RY$ the reversed sequence $RY = (y_{n-1},...,y_{1},y_{0})$ and let $H_{n}$ denote the
permutation automorphic subgroup $H_{n}=\{1,R\}\leq S_{n}\leq Aut(Z_{2}^{n})$. Hence $H_{n}$ defines
a partition on $\Z_{2}^{n}$ and $\mathfrak{S}(\Z_{2}^{n},H_{n})$ is a schur ring.

\begin{definition}
Let $Y\in \Z_{2}^{n}$. We shall call $Y$ symmetric if $RY=Y$ and otherwise we say it is 
non symmetric. We make $Sym(\Z_{2}^{n})$ the set of all $Y$ symmetric and $\widehat{Sym}(\Z_{2}^{n})$
the set of all $Y$ nonsymmetric.
\end{definition}

Take $Y\in\Z_{2}^{n}$ such that $Y$ is of the form $C^{i_{1}}XC^{i_{2}}X\cdots C^{i_{r}}X$. We want
to understand the structure of the words in $Sym(\Z_{2}^{n})$. As it must be fulfilled that $RY=Y$,
then taking $X=+\cdots+-+\cdots+$ in $\G_{n}(n-1)$ with $n$ an odd number we have

\begin{eqnarray*}
RY&=&RC^{i_{1}}XRC^{i_{2}}X\cdots RC^{i_{r}}X\\
&=&C^{-i_{1}}RXC^{-i_{2}}RX\cdots C^{-i_{r}}RX\\
&=&C^{n-i_{1}}XC^{n-i_{2}}X\cdots C^{n-i_{r}}X
\end{eqnarray*}

where we have used that $RX=X$. Hence if $Y$ is symmetric, then must be $n-i_{j}=i_{k}$ for $j\neq k$
ranging in $[1,r]$. Thereby $Y$ has the form 

\begin{equation}
Y_{0}^{\epsilon_{0}}Y_{1}^{\epsilon_{1}}\cdots Y_{(n-1)/2}^{\epsilon_{(n-1)/2}}
\end{equation}
for $\epsilon_{i}=0,1$ and $Y_{0}=X$ and $Y_{i}=C^{i}XC^{n-i}X$. As $C^{i}XC^{n-i}X$ is in
$Sym(\Z_{2}^{n})$ for all $i$, then
\begin{equation}
\mathcal{X}_{Sym^{O}}=\{X,\ CXC^{n-1}X,\ C^{2}XC^{n-2}X,...,\ C^{(n-1)/2}XC^{(n+1)/2}X\}
\end{equation}
is a $H_{n}$-code for $Sym(\Z_{2}^{n})$. For the case $n$ an even number it is easily followed that
\begin{equation}
\mathcal{X}_{Sym^{E}}=\{XC^{n-1}X,\ CXC^{n-2}X,\ C^{2}XC^{n-3}X,...,\ C^{(n-2)/2}XC^{n/2}X\}
\end{equation}
is a $H_{n}$-code for $Sym(\Z_{2}^{n})$, where $X=+++\cdots++-$. Also it is clear that 
$\mathcal{X}_{Sym^{E}}$ and $\mathcal{X}_{Sym^{O}}$ are $P(T)$-codes, therefore 
$Sym(\Z_{2}^{n})$ is an free $S$-subgroup in $\mathfrak{S}(\Z_{2}^{n},H_{n})$ and 
$Sym(\Z_{2}^{2n+1})$ and $Sym(\Z_{2}^{2n})$ have order $2^{n+1}$ and $2^{n}$, respectively.

Finally, the relationship between the symmetric subgroup $Sym(\Z_{2}^{^{2n+1}})$ and the 
$\delta_{2n}$-invariant $S$-subgroup $\mathbb{I}_{2n+1}(2n)$ is shown

\begin{theorem}
$Sym(\Z_{2}^{^{2n+1}})=\mathbb{I}_{2n+1}(2n)$.
\end{theorem}
\begin{proof}
By proposition \ref{prop_menos_uno_inv} the codewords in $\X_{\mathbb{I}_{2n+1}(2n)}$ are $X$ 
and $C^{q}XC^{2n+1-q}X$, $1\leq q\leq n$, with $X=+\cdots+-+\cdots+$. We want to show that all of
codewords in $\X_{\mathbb{I}_{2n+1}(2n)}$ is symmetric. For this we note that $RX=X$ and
\begin{eqnarray*}
R(C^{q}XC^{2n+1-q}X)&=&RC^{q}XRC^{2n+1-q}X\\
&=&C^{2n+1-q}RXC^{q}RX\\
&=&C^{2n+1-q}XC^{q}X.
\end{eqnarray*}
\end{proof}

\section{Schur ring $\mathfrak{S}(\Z_{2}^{n},H_{n}C_{n})$}

In this section we will use the commutation relation
\begin{equation}\label{comm_RC}
RC^{i}=C^{n-i}R
\end{equation}
to show that the $S$-subgroups $\mathbb{G}_{d}(n)$ and $Sym(\Z_{2}^{n})$ are $S$-subgroups in
$\mathfrak{S}(\Z_{2}^{n},H_{n}C_{n})$.

\begin{theorem}\label{theo_G_HC}
$\mathbb{G}_{dC}(n)$ is an $S$-subgroup in $\mathfrak{S}(\Z_{2}^{n},H_{n}C_{n})$
\end{theorem}
\begin{proof}
By having in mind the commutation relation (\ref{comm_RC}), we can to show that 
$\mathcal{X}_{F,d}$ is an $H_{n}$-code. In this way we have
\begin{eqnarray*}
RA_{i,d}X&=&R(C^{i}XC^{i+d}X\cdots C^{i+n-d}X)\\
&=&RC^{i}XRC^{i+d}X\cdots RC^{i+n-d}X\\
&=&C^{n-i}RXC^{n-i-d}RX\cdots C^{d-i}RX
\end{eqnarray*}
As $X=-++\cdots++$, then $RX=CX$. Hence
\begin{eqnarray*}
RA_{i,d}X&=&C^{n-i}CXC^{n-i-d}CX\cdots C^{d-i}CX\\
&=&C(C^{n-i}XC^{n-i-d}X\cdots C^{d-i}X)
\end{eqnarray*}
Finally, reordering and rewriting 
\begin{eqnarray*}
RA_{i,d}X&=&C(C^{d-i}X\cdots C^{d-i+(n-2d)}XC^{d-i+(n-d)}X)\\
&=&CA_{d-i,d}X\\
&=&A_{d-i+1,d}X.
\end{eqnarray*}
\end{proof}

\begin{definition}
A $S$-set $Y_{C}$ in $\Z_{2C}^{n}$ is symmetric if $R\cdot Y_{C}=Y_{C}$, where $R\cdot Y_{C}$ means 
the action of $R$ on the elements of $Y_{C}$. The set of all symmetric $S$-sets will be denoted by
$Sym(\Z_{2C}^{n})$ and the set of all non-symmetric $S$-sets will be denoted by
$\widehat{Sym}(\Z_{2C}^{n})$.
\end{definition}

It is easy to note that $R\cdot Y_{C}=(RY)_{C}$. Therefore $Y_{C}$ is a symmetric $S$-set if and 
only if contains some $C^{i}Y$ symmetric.

\begin{theorem}
$Sym(\Z_{2C}^{n})$ is an $S$-subgroup of $\mathfrak{S}(\Z_{2}^{n},H_{n}C_{n})$.
\end{theorem}
\begin{proof}
Take $Y_{C},Z_{C}$ in $Sym(\Z_{2C}^{n})$ and suppose that $Y,Z\in Sym(\Z_{2}^{n})$. As 
$$R(YC^{k}Z)_{C}=(RYRC^{k}Z)_{C}=(YC^{n-k}Z)_{C},$$
then $R(Y_{C}Z_{C})=R(Y_{C})R(Z_{C})=Y_{C}Z_{C}$.
\end{proof}

\section{Schur ring $\mathfrak{S}(\Z_{2}^{n},\Delta_{n}C_{n})$}

In this section we will use the commutation relation
\begin{equation}\label{comm_DC}
\delta_{a}C^{i}=C^{ia^{-1}}\delta_{a}
\end{equation}
to show that the $S$-subgroups $\mathbb{G}_{d}(n)$ and $\mathbb{I}_{n}(a)$ are $S$-subgroups in
$\mathfrak{S}(\Z_{2}^{n},\Delta_{n}C_{n})$.

\begin{theorem}\label{theo_G_DC}
$\mathbb{G}_{dC}(n)$ is an $S$-subgroup in $\mathfrak{S}(\Z_{2}^{n},\Delta_{n}C_{n})$.
\end{theorem}
\begin{proof}
We want to show that $\X_{F,d}$ is a $\Delta_{n}$-code. Take $A_{i,d}X$ in $\X_{F,d}$. Then if
$k=\frac{n}{d}-1$ we have
\begin{eqnarray*}
\delta_{a}A_{i,d}X&=&\delta_{a}(C^{i}XC^{i+d}X\cdots C^{i+kd}X)\\
&=&\delta_{a}C^{i}X\delta_{a}C^{i+d}X\cdots\delta_{a}C^{i+kd}X)\\
&=&C^{ia^{-1}}\delta_{a}XC^{(i+d)a^{-1}}\delta_{a}X\cdots C^{(i+kd)a^{-1}}\delta_{a}X\\
&=&C^{ia^{-1}}XC^{(i+d)a^{-1}}X\cdots C^{(i+kd)a^{-1}}X.
\end{eqnarray*}
Define the map $\vartheta:\X_{F,d}\rightarrow\{ni/d:\ i=1,2,...,d-1\}$ by  
\begin{eqnarray*}
\vartheta(A_{i,d}X)&=&\vartheta(C^{i}XC^{i+d}X\cdots C^{i+kd}X)\\
&=&\sum_{j=0}^{k}(i+jd)\\
&=&i(k+1)+\frac{k(k+1)}{2}d\\
&=&\frac{ni}{d}+\left(\frac{n}{2d}-\frac{1}{2}\right)n\\
&\equiv&\frac{ni}{d}\mod n
\end{eqnarray*}
Then $\vartheta$ is a biyection. As $\vartheta(\delta_{a}A_{i,d}X)\equiv\frac{a^{-1}ni}{d}\mod n$
and $\vartheta(C^{l}\delta_{a}A_{i,d}X)\equiv(a^{-1}+l)\frac{ni}{d}\mod n$, is followed that
$\delta_{a}A_{i,d}X\in\X_{F,d}$ and therefore $\X_{F,d}$ is a $\Delta_{n}$-code.
\end{proof}

\begin{definition}
Let $a$ be a unit in $\Z_{n}^{*}$. A basic set $Y_{C}$ in $\Z_{2C}^{n}$ is
$\delta_{a}$-\textbf{invariant} if $\delta_{a}\cdot Y_{C}=Y_{C}$. Denote by $\mathbb{I}_{nC}(a)$ 
the set of these $Y$.
\end{definition}

By having in mind the basic set $Y_{C}$, we can note that $\delta_{a}\cdot Y_{C}=(\delta_{a}Y)_{C}$.
Hence $Y_{C}$ is $\delta_{a}$-invariant in $\Z_{2C}^{n}$ if and only if contains some $C^{i}Y$ 
$\delta_{a}$-invariant in $\Z_{2}^{n}$.

\begin{theorem}
$\mathbb{I}_{nC}(a)$ is an $S$-subgroup in $\mathfrak{S}(\Z_{2}^{n},\Delta_{n}C_{n})$.
\end{theorem}
\begin{proof}
Take $Y_{C},Z_{C}$ in $\mathbb{I}_{nC}(a)$ and suppose that $Y,Z\in\mathbb{I}_{n}(a)$. As 
$\delta_{a}(YC^{k}Z)_{C}=(\delta_{a}Y\delta_{a}C^{k}Z)_{C}=(YC^{ka^{-1}}Z)_{C}$,
then $\delta_{a}(Y_{C}Z_{C})=\delta_{a}(Y_{C})\delta_{a}(Z_{C})=Y_{C}Z_{C}$.
\end{proof}

\section{Schur ring $\mathfrak{S}(\Z_{2}^{n},H_{n}\Delta_{n}C_{n})$}

Finally we show that the $S$-subgroups $\mathbb{G}_{d}(n)$, $Sym(\Z_{2}^{n})$ and 
$\mathbb{I}_{n}(a)$ are $S$-subgroups in $\mathfrak{S}(\Z_{2}^{n},H_{n}\Delta_{n}C_{n})$.
Let $Y_{C}$ be any basic set in $\Z_{2C}^{n}$. It is a very easy to notice that
$Y_{C}=(C^{k}Y)_{C}$ for all $k$. We will use this fact for to prove the following lemma

\begin{lemma}
$_{}$
\begin{enumerate}
\item $\delta_{a}Y_{iC}=Y_{a^{-1}iC}$ for $Y_{i}=C^{i}XC^{2n+1-i}X\in\X_{Sym^{O}}$.
\item $\delta_{a}Y_{iC}=Y_{\left(a^{-1}i+\frac{a^{-1}-1}{2}\right)C}$ for 
$Y_{i}=C^{i}XC^{2n-1-i}X\in\X_{Sym^{E}}$.
\end{enumerate}
\end{lemma}
\begin{proof}
$1.$ Take $\delta_{a}$ in the group $\Delta_{n}$. It is clear that $\delta_{a}X=C^{k_{a}}X$ for 
some $k_{a}$ depending on $a$, where  $X=+\cdots+-+\cdots+$ is the word in $\X_{C}$ used to 
construct all codewords in $\mathcal{X}_{Sym^{O}}$. Then
\begin{eqnarray*}
\delta_{a}Y_{iC}&=&(\delta_{a}C^{i}X\delta_{a}C^{2n+1-i}X)_{C}\\
&=&(C^{a^{-1}i}\delta_{a}XC^{a^{-1}(2n+1-i)}\delta_{a}X)_{C}\\
&=&(C^{k_{a}}(C^{a^{-1}i}XC^{2n+1-a^{-1}i}X))_{C}\\
&=&(C^{a^{-1}i}XC^{2n+1-a^{-1}i}X)_{C}\\
&=&Y_{a^{-1}iC}.
\end{eqnarray*}
$2.$ Equally, $\delta_{a}X=C^{k_{a}}X$ for some $k_{a}$ depending on $a$, where $X=++\cdots++-$ 
in $\X_{C}$ is used to construct all codewords in $\mathcal{X}_{Sym^{E}}$. Then
\begin{eqnarray*}
\delta_{a}Y_{iC}&=&(\delta_{a}C^{i}X\delta_{a}C^{2n-1-i}X)_{C}\\
&=&(C^{a^{-1}i}\delta_{a}XC^{a^{-1}(2n-1-i)}\delta_{a}X)_{C}\\
&=&(C^{k_{a}}(C^{a^{-1}i}XC^{2n-a^{-1}(i+1)}X))_{C}\\
&=&\left(C^{\frac{a^{-1}-1}{2}}(C^{a^{-1}i}XC^{2n-a^{-1}(i+1)}X)\right)_{C}\\
&=&\left(C^{a^{-1}i+\frac{a^{-1}-1}{2}}XC^{2n-a^{-1}(i+1)+\frac{a^{-1}-1}{2}}X\right)_{C}\\
&=&Y_{\left(a^{-1}i+\frac{a^{-1}-1}{2}\right)C}.
\end{eqnarray*}
\end{proof}

We will use this lemma for to show that the symmetric binary sequences form an $S$-subgroup in
$\mathfrak{S}(\Z_{2}^{n},H_{n}\Delta_{n}C_{n})$

\begin{theorem}\label{theo_S_subgroup_Sym_HDC}
$Sym(\Z_{2C}^{n})$ is an $S$-subgroup in $\mathfrak{S}(\Z_{2}^{n},H_{n}\Delta_{n}C_{n})$.
\end{theorem}
\begin{proof}
Clearly $Sym(\Z_{2C}^{n})$ is an $S$-subgroup of $\mathfrak{S}(\Z_{2}^{n},H_{n}C_{n})$. From 
the previous lemma is followed that $\Delta_{n}$ defines a partition on $\X_{Sym^{E}}$ and 
$\X_{Sym^{O}}$. Hence they are $\Delta_{n}$-codes and $Sym(\Z_{2C}^{n})$ is an $S$-subgroup of 
$\mathfrak{S}(\Z_{2}^{n},H_{n}\Delta_{n}C_{n})$.
\end{proof}

\begin{theorem}
$\mathbb{G}_{dC}(n)$ is an $S$-subgroup in $\mathfrak{S}(\Z_{2}^{n},H_{n}\Delta_{n}C_{n})$.
\end{theorem}
\begin{proof}
Follows from theorems \ref{theo_G_HC} and \ref{theo_G_DC}.
\end{proof}

\begin{theorem}
$\mathbb{I}_{nC}(a)$ is an $S$-subgroup in $\mathfrak{S}(\Z_{2}^{n},H_{n}\Delta_{n}C_{n})$.
\end{theorem}
\begin{proof}
From theorem \ref{theo_S_subgroup_Sym_HDC} the case $a=n-1$ is 
excluded. In the previous section already was proved that $\mathbb{I}_{nC}(a)$ is an $S$-subgroup in
$\mathfrak{S}(\Z_{2}^{n},\Delta_{n}C_{n})$. Now, we wish to show that $H_{n}$ defines a partition
on $\mathbb{I}_{nC}(a)$ by using (\ref{comm_RC}). Take the codeword $\mathsf{C}_{s}X$ in 
$\X_{\mathbb{I}_{n}(a)}$. We have then
\begin{eqnarray*}
R(\mathsf{C}_{s}X)_{C}&=&R(C^{s}XC^{sa}X\cdots C^{sa^{t_{s}-1}}X)_{C}\\
&=&(RC^{s}XRC^{sa}X\cdots RC^{sa^{t_{s}-1}}X)_{C}\\
&=&(C^{n-s}CXC^{n-sa}CX\cdots C^{n-sa^{t_{s}-1}}CX)_{C}\\
&=&(C(C^{n-s}XC^{n-sa}X\cdots C^{n-sa^{t_{s}-1}}X))_{C}\\
&=&(C^{n-s}XC^{(n-s)a}X\cdots C^{(n-s)a^{t_{s}-1}}X))_{C}\\
&=&(\mathsf{C}_{n-s}X)_{C}.
\end{eqnarray*}
\end{proof}

\Addresses

\end{document}